\documentclass[english,12pt]{amsart} 

\usepackage[T1]{fontenc} \usepackage{ucs} \usepackage[top=1in,
bottom=1in, left=1in, right=1in]{geometry} \usepackage{color}

\usepackage{amsmath} \usepackage{amsfonts,dsfont}
\usepackage[psamsfonts]{amssymb}
\usepackage{mathtools} 
\usepackage{amsthm} \usepackage{mathrsfs} 
\usepackage{graphicx}
\usepackage{enumerate} 
\usepackage{hyperref,soul} \usepackage{soul} \usepackage{comment}
\usepackage{mathabx}
\usepackage{ulem}

 \newtheorem{theoA}{Theorem}

\newtheorem{propA}[theoA]{Proposition}

\newtheorem{thm}{Theorem} 
\newtheorem{lem}{Lemma}
\newtheorem{prop}{Proposition} 
\newtheorem{coro}{Corollary}
\newtheorem{rmk}{Remark}

\newtheorem{expl}{Example}

\usepackage{thmtools}
\declaretheoremstyle[notefont=\bfseries,notebraces={}{},%
headpunct={},postheadspace=1em]{mystyle}
\declaretheorem[style=mystyle,numbered=no,name=Theorem]{thm-hand}
\newcommand{\eqnsection}{
  \makeatletter \csname @addtoreset\endcsname{equation}{section}
  \makeatother} \eqnsection

\graphicspath{{./images/}
} 
\declaretheoremstyle[notefont=\bfseries,notebraces={}{},%
headpunct={},postheadspace=1em]{mystyle}

 \def\E{{\mathbb{E}}} \def\R{{\mathbb{R}}}
\def\N{{\mathbb{N}}}  \def\demi{{1\over 2}}
\def\ggg{{\mathcal G}}  

\def\Z{{\mathbb Z}}
\def\Q{{\mathbb Q}}
\def\ttt{{\mathcal T}}
\def\mmm{{\mathcal {M}}}
\def\var{{\operatorname{Var}}}
\def\eee{{\mathcal{E}}}
\def\dive{{\operatorname{div}}}
\def\indic{{\bf 1}}
\def\lll{{\mathcal L}}

\begin{document}
\author[C. Sabot]{Christophe SABOT}
\address{Universit\'e de Lyon, Universit\'e Lyon 1,
Institut Camille Jordan, CNRS UMR 5208, 43, Boulevard du 11 novembre 1918,
69622 Villeurbanne Cedex, France} \email{sabot@math.univ-lyon1.fr}
\title[Polynomial localization of the 2D-VRJP]{Polynomial localization of the 2D-Vertex Reinforced Jump Process}
\begin{abstract}
We prove polynomial decay of the mixing field of the Vertex Reinforced Jump Process (VRJP) on $\Z^2$ with bounded conductances. Using \cite{SZ15} we deduce that the VRJP on $\Z^2$ with any constant conductances is almost surely recurrent. It gives a counterpart of the result of Merkl, Rolles \cite{merkl2009recurrence} and Sabot, Zeng \cite{SZ15} for the 2-dimensional Edge Reinforced Random Walk.
\end{abstract}
\thanks{This work was supported by the LABEX MILYON (ANR-10-LABX-0070) of Universit\'e de Lyon, within the program "Investissements d'Avenir" (ANR-11-IDEX-0007) operated by the French National Research Agency (ANR), and by the ANR/FNS project MALIN (ANR-16-CE93-0003).} 
\maketitle
Let \(\ggg=(V,E)\) be an undirected graph with finite degree at each vertex. We note $i\sim j$ if $\{i,j\}$ is an edge of the graph.
Let \((W_{i,j})_{i\sim j}\) be a set of positive conductances on the edges, \(W_{i,j}>0\), \(W_{i,j}=W_{j,i}\).
The Vertex Reinforced Jump Process (VRJP) is the continuous time process \((Y_s)_{s\ge0}\) on \(V\), starting at time \(0\) at some vertex \(i_0\in V\),
which, conditionally on the past at time \(s\), if \(Y_s=i\),   jumps to a neighbour \(j\) of \(i\) at rate
\[W_{i,j}L_j(s),\]
where
\[L_j(s):=1+\int_0^s \mathds{1}_{\{Y_u=j\}}\,du.\]

The VRJP was introduced by Davis and Volkov and investigated on $\Z$ in \cite{davis2004vertex}, then on trees in \cite{collevecchio2009limit,basdevant2012continuous}. In~\cite{ST15}, Sabot and Tarr\`es proved that this process is closely related to the Edge Reinforced Random Walk (ERRW), and that on any finite graphs, after some some time-change, it is a mixture of Markov jump processes, the mixing law being the first marginal of the supersymmetric hyperbolic sigma field introduced by 
Disertori, Spencer, Zirnbauer \cite{Zirnbauer91, DSZ06}. Using the exponential localization result of Disertori and Spencer \cite{DS10}, it was proved in~\cite{ST15} that on any graph with bounded degree, there exists a value $\underline W$ such that if $W_{i,j}\le \underline W$ for all $i\sim j$, the VRJP is positive recurrent, i.e. the VRJP visits infinitely often and spends a positive portion of the time on all point (an alternative proof of the localization of the VRJP was given by Angel, Crawford, Kozma in~\cite{angel2014localization} using the representation as a mixture proved in~\cite{ST15}). Using the delocalization result of Disertori, Spencer, Zirnbauer~\cite{DSZ06}, a phase transition was proved on $\Z^d$, $d\ge 3$: there exists $\overline W(d)$, such that if $W_{i,j}\ge\overline W(d)$ for all $i\sim j$, the VRJP is transient, and also diffusive for $(W_{i,j})$ constant and large enough (\cite{SZ15}).

Similar results hold for the Edge Reinforced Random Walk (ERRW) (see~\cite{diaconis1980finetti,coppersmith1987random,ST15,angel2014localization,DST14}). Besides, on $\Z^2$, a polynomial localization of the mixing field of the ERRW (the so-called magic formula of Coppersmith and Diaconis) was proved by Merkl and Rolles~\cite{merkl2009recurrence}. By itself, this polynomial localization does not entail recurrence of the ERRW (it was used in \cite{merkl2009recurrence} to prove recurrence of the ERRW on a modification of $\Z^2$ at weak reinforcement). However, together with the representation of the VRJP and ERRW on infinite graphs as mixture of Markov jump processes provided in \cite{STZ15,SZ15}, it allows to prove recurrence of the ERRW on $\Z^2$ for all initial constants weights.  

The aim of this paper is to provide a counterpart to the result of Merkl and Rolles \cite{merkl2009recurrence}, i.e. to prove polynomial decay of the mixing field of the VRJP. As remarked in \cite{SZ15}, it implies recurrence of the VRJP with constant conductances on $\Z^2$, in the sense that any point is a.s. visited infinitely often by the VRJP. The proof is in the spirit of the proof of Merkl and Rolles for the ERRW (and much before in the spirit of the argument of Mac Bryan and Spencer \cite{McBryan-Spencer77} for the $SO(N)$ symmetric ferromagnets), based on a deformation of the field by a deterministic harmonic function.

\section{Statement of the results}
\subsection{The mixing field of the VRJP}
We first recall how the VRJP can be written as a mixture of Markov jump processes and its relation with the first marginal of the supersymmetric hyperbolic sigma model.

We denote by $\vec{E}$ the set of corresponding directed edges associated with the undirected edges $E$ (i.e. with each edge of $E$ we associated two edges with opposite orientations). We denote
$$\sum_{i\to j} \boldsymbol{\cdot} = \sum_{(i,j)\in \vec{E}} \boldsymbol{\cdot}
$$
the sum on directed edges of the network. For a function $u:V\mapsto \R$ and for $(i, j)\in \vec{E}$, we denote the gradient of $u$ on $(i,j)$ by:
$$
\nabla u_{i,j}:= u_j-u_i.
$$

Assume $V$ is finite. We introduce the mixing field of the VRJP. For a fixed set of positive conductances $(W_{i,j})_{\{i,j\}\in E}$, and a vertex $i_0\in V$, we denote by $\Q^W_{i_0}(du)$ the positive measure on $\{(u_i)_{i\in V}, \;\; u_{i_0}=0\}$ defined by
\begin{eqnarray}\label{00}
\Q_{i_0}^W(du)=c_V e^{-\demi\sum_{i\to j} W_{i,j}(e^{\nabla u_{i,j}}-1)} \sqrt{D_{i_0}(W,u)} (\prod_{i\neq i_0} du_i),
\end{eqnarray}
where $c_V=1/\sqrt{2\pi}^{\vert V\vert -1}$, 
and
$$
D_{i_0}(W,u)=\sum_{T \in \ttt_{i_0}}\prod_{(i,j)\in T} W_{i,j} e^{\nabla u_{i,j}},
$$
where $ \ttt_{i_0}$ is the set of directed spanning trees oriented towards the root $i_0$.
(The choice of directed spanning trees with weights $e^{u_j-u_i}$, instead of $e^{u_i+u_j}$ classically, explains that the integration is with respect to the measure $(\prod_{i\neq i_0} du_i)$, instead of $(\prod_{i\neq i_0} e^{-u_i} du_i)$ classically.) 


The following fact was initially proved in \cite{DSZ06} by supersymmetric arguments, then in \cite{ST15} by probabilistic arguments and in \cite{STZ15} by direct computation.
\begin{theoA}
The measure $\Q^W_{i_0}(du)$ is a probability measure on the set $\{(u_i)_{i\in V}, \;\; u_{i_0}=0\}$.
\end{theoA}
For simplicity, we will often write $\E^{\Q^W_{i_0}}(\cdot)$ for $\int \cdot\; \Q^W_{i_0}(du)$. The following is a simple consequence of the previous theorem.
\begin{coro}\label{cor1}
For any $i_0, j_0\in V$:
$$
\E^{\Q^W_{i_0}}(e^{u_{j_0}})=1.
$$
\end{coro}
\begin{proof}
By simple computation, changing from variable $(u_i)$ to $(\tilde u_i)=(u_i-u_{i_0})$, we get that
$$
\int e^{u_{j_0}}\Q^W_{i_0}(du)=\int \Q^W_{j_0}(d\tilde u)=1.$$
\end{proof}
The following result relate the mixing field $\Q^W_{i_0}(du)$ with the VRJP and was proved in \cite{ST15}.
\begin{theoA}
After some time change (see \cite{ST15} for details), the VRJP starting from $i_0\in V$ with conductances $(W_{i,j})_{i\sim j}$ is a mixture of Markov jump processes with jump rates 
$\demi W_{i,j}e^{U_j-U_i}$, where $(U_i)_{i\in V}$ is distributed according to $\Q^W_{i_0}(du)$. More precisely, we have the following identity of distributions:
$$
\lll_{i_0}^{VRJP}(\cdot)=\int \lll^{(u)}_{i_0} (\cdot) \Q^W_{i_0}(du),
$$
where $\lll_{i_0}^{VRJP}$ is the law of the (time-changed) VRJP starting from $i_0$ and $\lll^{(u)}_{i_0}$ is the law of the Markov jump process starting from $i_0$ and with jump rate from i to $j\sim i$,
$$
\demi W_{i,j}e^{u_j-u_i}.
$$
\end{theoA}

\subsection{Main results}
\label{main_results}
 We focus now on the lattice $\Z^2$ and its restriction to finite boxes. We denote by $\ggg_{\Z^2}=(\Z^2, E_{\Z^2})$ the usual $\Z^2$ lattice where $\{i,j\}\in E_{\Z^2}$ if $\vert i-j\vert_{1}=1$. We assume that the lattice is endowed with some positive conductances $(W_{i,j})_{i\sim j}$.

For $N$ a positive integer, we set $V_N:=\Z^2\cap [-N,N]^2$, and denote by $\ggg_N$ the restriction of $\ggg_{\Z^2}$ to $[-N,N]^2$ with wired boundary condition. More precisely, $\ggg_N:=(\tilde V_N, \tilde E_n)$ where $\tilde V_N:= V_N\cup\{\delta_N\}$ and $\tilde E_N$ are obtained by contracting all the vertices of $\Z^2\setminus V_N$ to the single point $\delta_N$ (the edges are obtained as the image of the edges of $\ggg_{\Z^2}$ by this contraction and by removing all the loops created and identifying multiple edges). The graph $\ggg_N$ is naturally endowed with the conductances $(W^N_{e})_{e\in \tilde E_N}$ obtained by this restriction: the conductance of an edge is the sum of the conductances of the edges of $E_{\Z^2}$ mapped to it by the contraction. (See \cite{SZ15}, section~4.1 for details of the construction.) The estimates below are also valid for the free wired boundary condition, but the wired boundary condition is useful for the application to recurrence. We denote by $\Q^N_{i_0}$ the mixing field associated with this graph with conductances $(W^N_{e})_{e\in \tilde E_N}$ and simply by $\Q^N$ when $i_0=0$.

The main theorem proves polynomial decay of some exponential moments of the mixing field under $\Q^N(du)$.
\begin{thm}\label{main_theorem}
Assume that the conductances are uniformly bounded: $W_{i,j}\le \overline W<\infty $ for all $i\sim j$, $i,j\in \Z^2$. Then, for $0<s<1$, there exists $\eta=\eta(\overline W,s)>0$ such that for all $N\in \N$ large enough, for all $y\in V_N$,
$$
\E^{\Q^N}\left(e^{s u_y}\right)\le \vert y\vert ^{-\eta}.
$$
\end{thm} 
\begin{rmk} An explicite expression is provided for $\eta$, see~\eqref{explicite}.
\end{rmk}
As stated in Remark~7 of~\cite{SZ15}, such an estimate implies that the VRJP is recurrent on $\Z^2$.
\begin{thm}
On the graph $\Z^2$ with constant conductances on horizontal edges and on vertical edges, the VRJP is recurrent, i.e. almost surely, the VRJP visits infinitely often every point.
\end{thm}
\begin{rmk} A weaker version of the recurrence was proved for the 2D-VRJP by Bauerschmidt, Helmuth and Swan in~\cite{BHS18}: their result asserts that the expectation of the total time spent at the origin is infinite. Their approach is based on a direct relation between the VRJP at finite time and the full supersymmetric hyperbolic sigma model and by an adaptation of the original Mermin-Wagner argument.
\end{rmk} 
\begin{proof} The proof of the theorem is the same as the proof of the corresponding theorem for the ERRW, see Theorem~5 of~\cite{SZ15}. In \cite{SZ15}, a stationary ergodic function $(\psi(i))_{i\in \Z^2}$ is constructed, which is a.s. equal to 0 if and only if the VRJP is recurrent. The polynomial decay of the mixing field $\E^{\Q^N}\left(e^{su_y}\right)$ implies that the function $\psi$ is equal to $0$ and thus that the VRJP is recurrent.
\end{proof}

\section{Proof of Theorem~\ref{main_theorem}}
\subsection{An a priori estimate}
The proof is based on the following Mermin-Wagner type estimate. This estimate is valid for any finite graph $\ggg=(V,E)$ with conductances $(W_{i,j})_{i\sim j}$.
\begin{lem}\label{Lemma1}
Let $i_0$ and $y$ be two distinct vertices. Let $v:V\mapsto \R$ be such that $v(i_0)=0$,
$v(y)=1$.
For $0<s<1$, let $q>1$ be such that $s+{1\over q}=1$. Let $\gamma >0$ be such that 
\begin{eqnarray}\label{0}
q^2\gamma\vert \nabla v_{i,j}\vert \le \demi, \;\;\;Ê\forall  i\sim j \hbox{ in V}.
\end{eqnarray}
Then,
$$
\E^{\Q_{i_0}^W}\left(e^{s u_y}\right)\le e^{-\gamma s+\gamma^2 q^2 \sum_{i\to j} (W_{i,j}+1)\vert \nabla v_{i,j}\vert^2}.
$$
\end{lem}
In order to simplify the notations, we will simply write $\Q(du)$ for $\Q_{i_0}^W(du)$ and $D(W,u)$ for $D_{i_0}(W,u)$.
\begin{proof}
We start by a simple change of variables.
\begin{prop}\label{prop_chang}
For $\gamma\in \R$ we denote by $\Q^\gamma$ the distribution of $\tilde u^\gamma := u-\gamma v$ when $u$ is distributed under $\Q(du)$. We have
$$
{d\Q\over d\Q^\gamma}(u) =e^{\demi\sum_{i\to j} W_{i,j} e^{\nabla u_{i,j}}(e^{\gamma\nabla v_{i,j}}-1)} \sqrt{{D(W,u)\over D(W, u+\gamma v)}},
$$
\end{prop}
\begin{proof}
If $\phi$ is a positive test function, by changing from variable $u$ to $\tilde u:=u-\gamma v$,
\begin{eqnarray*}
\int \phi(u-\gamma v) \Q(du)&=&
c_V \int \phi(u-\gamma v) e^{-\demi \sum_{i\to j} W_{i,j} (e^{\nabla u_{i,j}}-1)} \sqrt{{D(W,u)}} du
\\
&=&
c_V \int \phi(\tilde u) e^{-\demi \sum_{i\to j} W_{i,j} (e^{\nabla \tilde u_{i,j}+\gamma \nabla v_{i,j}}-1)} \sqrt{{D(W,\tilde u+\gamma v)}}d\tilde u
\\
&=&
\int \phi(\tilde u)  e^{-\demi\sum_{i\to j} W_{i,j} e^{\nabla \tilde u_{i,j}}(e^{\gamma\nabla v_{i,j}}-1)} \sqrt{{D(W,\tilde u+ \gamma v)\over D(W, \tilde u)}} \Q(d\tilde u)
\end{eqnarray*}

\end{proof}
Let us now prove the Lemma. We have by Corollary~\ref{cor1}
$$
\E^{\Q^\gamma}\left(e^{u_y}\right)=
\E^\Q\left(e^{u_y-\gamma v_y}\right)=e^{-\gamma} \E^\Q\left(e^{u_y}\right)=e^{-\gamma}.
$$
On the other hand, by H\"older inequality,
\begin{eqnarray}\label{1}
\nonumber 
\E^{\Q}\left(e^{s u_y}\right)= \E^{\Q^\gamma}\left({d\Q\over d\Q^\gamma} e^{s u_y}\right)&\le& \E^{\Q^\gamma}\left( \left({d\Q\over d\Q^\gamma}\right)^q \right)^{1/q} \E^{\Q^\gamma}\left(e^{u_y}\right)^s
\\
&\le&
e^{-\gamma s}\E^{\Q^\gamma}\left( \left({d\Q\over d\Q^\gamma}\right)^q \right)^{1/q}  
\end{eqnarray}
(It will be clear later that everything is integrable on the right-hand-side.)

Let us fix $\gamma'$ such that 
$$
\gamma':=-\gamma(q-1).
$$
We have,
\begin{eqnarray}\label{2}
\E^{\Q^\gamma}\left( \left({d\Q\over d\Q^\gamma}\right)^q \right)=\E^{\Q}\left( \left({d\Q\over d\Q^\gamma}\right)^{q-1} \right)=
\E^{\Q^{\gamma'}}\left( \left({d\Q\over d\Q^\gamma}\right)^{q-1} \left( {d\Q\over d\Q^{\gamma'}}\right)\right)
\end{eqnarray}
Then,  by Proposition~\ref{prop_chang}
\begin{eqnarray}
\label{3}
&&
\left({d\Q\over d\Q^\gamma}\right)^{q-1} \left( {d\Q\over d\Q^{\gamma'}}\right)(u)
\\
\nonumber
&=&
 e^{\demi\sum_{i\to j} W_{i,j} e^{\nabla u_{i,j}}((q-1)e^{\gamma\nabla v_{i,j}}+e^{\gamma'\nabla v_{i,j}}-q)} {\sqrt{D(W,u)}^q\over \sqrt{ D(W, u+\gamma v)}^{q-1} \sqrt{D(W,u+\gamma' v)}}
\\
\nonumber 
&=&
\exp\left( {q\over 2} \sum_{i\to j} W_{i,j} e^{\nabla u_{i,j}+\gamma' \nabla v_{i,j} }\left((1-{1\over q})e^{q\gamma\nabla v_{i,j}}+{1\over q} - e^{(q-1)\gamma\nabla v_{i,j}}\right)\right)
\\
\nonumber 
&&\cdot \exp\left( {q\over 2} \left(\ln D(W,u)-(1-{1\over q})\ln D(W,u+\gamma v)-{1\over q} \ln D(W,u+\gamma' v)\right)\right).
\end{eqnarray}
Let us consider the first line of the last expression: we make a second order expansion of the term $(1-{1\over q})e^{q\gamma\nabla v_{i,j}}+{1\over q} - e^{(q-1)\gamma\nabla v_{i,j}}$.
The constant term vanishes, and the first order is
$$
(1-{1\over q})q\gamma\nabla v_{i,j}-(q-1)\gamma\nabla v_{i,j}=0
$$
Hence we can bound by Taylor expansion:
\begin{eqnarray}
\nonumber
&&
\left \vert
(1-{1\over q})e^{q\gamma\nabla v_{i,j}}+{1\over q} - e^{(q-1)\gamma\nabla v_{i,j}}
\right\vert
\\
\nonumber
&\le&
\demi (q\gamma\nabla v_{i,j})^2 (1-{1\over q})e^{q\gamma\vert \nabla v_{i,j}\vert} +\demi ((q-1)\gamma\vert \nabla v_{i,j}\vert)^2 e^{(q-1)\gamma\vert \nabla v_{i,j}\vert}
\\
\nonumber
&\le&
q^2\gamma^2\vert \nabla v_{i,j}\vert^2 e^{q\gamma\vert \nabla v_{i,j}\vert}
\\
\label{ineg-1}
&\le&
2 q^2\gamma^2\vert \nabla v_{i,j}\vert^2
\\
\label{ineg-2}
&\le&
\demi
\end{eqnarray}
where \eqref{ineg-1} and \eqref{ineg-2} comes from the fact that $q\gamma\vert \nabla v_{i,j}\vert\le q^2\gamma\vert \nabla v_{i,j}\vert\le \demi$ by assumption \eqref{0}, and that $e^\demi\le 2$.

Concerning the second term we will use the following lemma.
\begin{lem}
The application $\gamma\to \ln D(W,u+\gamma v)$ is convex.
\end{lem}
\begin{rmk} The property was already remarked in \cite{DSZ06}, remark~2.3, and a similar statement was proved in the case of the ERRW, see the proof of Lemma~6.2 in \cite{merkl2009recurrence}.\end{rmk}
\begin{proof}
We have
\begin{eqnarray*}
{\partial \over \partial \gamma} \ln D(W,u+\gamma v)&=&
{
\sum_{T \in \ttt_{i_0}}\left(\prod_{(i,j)\in T} W_{i,j} e^{\nabla u_{i,j}+\gamma \nabla v_{i,j}}\right)
\left(\sum_{(i,j)\in T} \nabla v_{i,j}\right)
\over
\sum_{T \in \ttt_{i_0}}\prod_{(i,j)\in T} W_{i,j} e^{\nabla u_{i,j}+\gamma \nabla v_{i,j}}
}
\\
\end{eqnarray*}
\begin{eqnarray*}
{\partial^2 \over \partial \gamma^2} \ln D(W,u+\gamma v)
&=&
{
\sum_{T \in \ttt_{i_0}}\left(\prod_{(i,j)\in T} W_{i,j} e^{\nabla u_{i,j}+\gamma \nabla v_{i,j}}\right)
\left(\sum_{(i,j)\in T} \nabla v_{i,j}\right)^2
\over
\sum_{T \in \ttt_{i_0}}\prod_{(i,j)\in T} W_{i,j} e^{\nabla u_{i,j}+\gamma \nabla v_{i,j}}
}
\\
&&
-
\left({
\sum_{T \in \ttt_{i_0}}\left(\prod_{(i,j)\in T} W_{i,j} e^{\nabla u_{i,j}+\gamma \nabla v_{i,j}}\right)
\left(\sum_{(i,j)\in T} \nabla v_{i,j}\right)
\over
\sum_{T \in \ttt_{i_0}}\prod_{(i,j)\in T} W_{i,j} e^{\nabla u_{i,j}+\gamma \nabla v_{i,j}}
}\right)^2
\end{eqnarray*}
Hence,
$$
{\partial^2 \over \partial \gamma^2} \ln D(W,u+\gamma v)=\var_{\mmm(W,u+\gamma v)}\left(\sum_{(i,j)\in T} \nabla v_{i,j}\right)\ge 0,
$$
where $\mmm(W,u+\gamma v)$ is the probability on $\ttt_{0}$ defined as the law on random directed spanning trees oriented toward $i_0$ for the weights on directed edges 
$(W_{i,j}e^{\nabla u_{i,j}+\gamma\nabla v_{i,j}})_{(i,j)\in \vec{E}}$, and $\var_{\mmm(W,u+\gamma v)}$ is the associated variance.
\end{proof}
As a consequence, since $(1-{1\over q})\gamma+{1\over q} \gamma'=0$, we have
$$
(1-{1\over q})\ln D(W,u+\gamma v)+{1\over q} \ln D(W,u+\gamma' v)-\ln D(W,u)\ge 0.
$$
Hence by \eqref{3} and \eqref{ineg-1}
$$
\left({d\Q\over d\Q^\gamma}\right)^{q-1} \left( {d\Q\over d\Q^{\gamma'}}\right)(u)\le 
\exp\left( {q\over 2} \sum_{i\to j} W_{i,j} e^{\nabla u_{i,j}+\gamma' \nabla v_{i,j} }\left(2 q^2\gamma^2\vert\nabla v_{i,j}\vert^2
\right)
\right)
$$
Hence,
\begin{eqnarray*}
&&\E^{\Q^{\gamma'}}\left( \left({d\Q\over d\Q^\gamma}\right)^{q-1} \left( {d\Q\over d\Q^{\gamma'}}\right)\right)
\\
&\le&
c_V \int 
\exp\left( -\demi \sum_{i\to j} W_{i,j}\left(e^{\nabla u_{i,j}+\gamma' \nabla v_{i,j} }-1-2 q^3\gamma^2\vert\nabla v_{i,j}\vert^2 e^{\nabla u_{i,j}+\gamma' \nabla v_{i,j} }
\right)
\right)
\sqrt{D(W,u+\gamma' v)} du
\\
&\le&
e^{\sum_{i\to j} W_{i,j} q^3\gamma^2\vert\nabla v_{i,j}\vert^2}
c_V \int 
\exp\left( -\demi \sum_{i\to j} \tilde W_{i,j}
\left(e^{\nabla u_{i,j}+\gamma' \nabla v_{i,j} }-1\right)
\right)
\sqrt{D(W,u+\gamma' v)} du
\end{eqnarray*}
with
$$
\tilde W_{i,j}:= W_{i,j}(1-2 q^3\gamma^2\vert\nabla v_{i,j}\vert^2)
$$
Remark that by assumption, we have $2 q^3\gamma^2\vert\nabla v_{i,j}\vert^2\le 2 (q^2\gamma\vert\nabla v_{i,j}\vert)^2\le \demi$, we have
$$
\tilde W_{i,j}\ge \demi W_{i,j}>0,
$$
hence the measure $\Q^{\tilde W}:=\Q^{\tilde W}_{i_0}$ defined by \eqref{00} with conductances $(\tilde W_{i,j})$ is well-defined as a probability.
Changing back to coordinate $\tilde u=u+\gamma' v$ we get that
\begin{eqnarray*}
&&\E^{\Q^{\gamma'}}\left( \left({d\Q\over d\Q^\gamma}\right)^{q-1} \left( {d\Q\over d\Q^{\gamma'}}\right)\right)
\\
&\le&
e^{\sum_{i\to j} W_{i,j} q^3\gamma^2\vert\nabla v_{i,j}\vert^2}
c_V \int 
\exp\left( -\demi \sum_{i\to j} \tilde W_{i,j}
\left(e^{\nabla \tilde u_{i,j}}-1\right)
\right)
\sqrt{D(W,\tilde u)} d\tilde u
\\
&=&
e^{\sum_{i\to j} W_{i,j} q^3\gamma^2\vert\nabla v_{i,j}\vert^2}
\int 
\sqrt{
{D(W,\tilde u)
\over D(\tilde W,\tilde u)}
} 
\Q^{\tilde W}(d\tilde u)
\end{eqnarray*}
Now, since $2 q^3\gamma^2\vert\nabla v_{i,j}\vert^2\le \demi$ and $(1-h)^{-1}\le e^{2h}$ if $0\le h\le \demi$,
$$
{D(W,\tilde u)
\over D(\tilde W,\tilde u)}
\le 
\prod_{\{i,j\}\in E} (1-2 q^3\gamma^2\vert\nabla v_{i,j}\vert^2)^{-1}\le \exp\left( \sum_{\{i,j\}\in E} 2q^3\gamma^2\vert\nabla v_{i,j}\vert^2\right)
\le \exp\left( \sum_{i\to j} q^3\gamma^2\vert\nabla v_{i,j}\vert^2\right).
$$
It follows that
$$
\E^{\Q^{\gamma'}}\left( \left({d\Q\over d\Q^\gamma}\right)^{q-1} \left( {d\Q\over d\Q^{\gamma'}}\right)\right)^{1/q}
\le
\exp\left( \sum_{i\to j} (W_{i,j}+1) q^2\gamma^2\vert\nabla v_{i,j}\vert^2\right).
$$
Together with \eqref{1} and \eqref{2}, it concludes the proof of the lemma. 
\end{proof}
\subsection{Back to the $\Z^2$ lattice}
We assume in this section that the graph is the graph $\ggg_N=(\tilde V_N,\tilde E_N)$ defined in Section~\ref{main_results}. We will apply the previous lemma in the case where $i_0=0$ and $y\in V_N$.

The next step to conclude the proof of Theorem~\ref{main_theorem} is to construct a good function $v$ which satisfies the hypothesis of Lemma~\ref{Lemma1} and with a good control on its $l^2$ norm. We denote by $\eee$ the Dirichlet form on the graph $\ggg_N$ with conductances 1 defined for $f:\tilde V_N\mapsto \R$ by
$$
\eee(f,f)= \demi\sum_{i\to j} \vert \nabla f_{i,j}\vert^2.
$$
Let $v$ be the harmonic function between 0 and $y\in V_N$, $y\neq 0$, for constant conductances 1:
$$
\begin{cases}
v(0)=0,
\\
v(y)=1,
\\
\sum_{j, j\sim z}  \nabla v_{z,j}=0, \;\;\; \forall z \in V_N, \; z\neq 0, \; z\neq y.
\end{cases}
$$ 
By definition
$$
\eee(v,v)={1\over R(0,y)},
$$
where $R(0,y)$ is the equivalent resistance between $0$ and $y$ for the graph $\ggg_N$ with unit conductances.
Classically, by Nash-William criterion, there exists $c_0>0$, independent of $N$ and $y$, such that
$$
R(0,y)\ge c_0 \ln \vert y\vert_\infty,
$$
see e.g.  \cite{lyons-peres}, formula~(2.7) Section~2.4 taking the annuli between 0 and $y$ as cut-sets. (Note that we can take $c_0$ arbitrary close to $1/8$ for $\vert y\vert_{\infty}$ large enough, since \cite{lyons-peres}, formula~(2.7) implies that $R(0,y)\ge \sum_{k=1}^{\vert y\vert_\infty-1}{1\over 4(2k+1)}\sim {1\over 8} \ln \vert y\vert_\infty$). Moreover we have,
$$
\dive(\nabla v)(z)= {1\over R(0,y)}(\indic_{z=0}-\indic_{z=y}).
$$
where $\dive(\nabla v)$ is the divergence of $\nabla v$ defined by $\dive(\nabla v)(z)=\sum_{j, j\sim z}  \nabla v_{z,j}$. This implies that $R(0,y)\nabla v$ is a unit flow between 0 and $y$, in fact it is the current flow, see \cite{lyons-peres}~Section~2.4. In particular it implies, by \cite{lyons-peres} Proposition~2.2 and exercise~2.37,  that 
$$
R(0,y)\vert \nabla v_{i,j}\vert\le 1, \;\;\; \forall i\sim j.
$$
Take 
\begin{eqnarray}\label{condition}
\gamma= \tilde \gamma R(0,y), \;\;\; \hbox{ with }\;\;\; \tilde \gamma\le {1\over 2q^2},
\end{eqnarray}
$\tilde \gamma$ to be fixed later.
We have that 
$$\gamma q^2\vert \nabla v_{i,j}\vert \le \tilde  \gamma  q^2 (R(0,y)\vert \nabla v_{i,j}\vert)\le \demi,$$
and $v$ satisfies the hypothesis of Lemma~\ref{Lemma1}.
Hence, we can apply Lemma~\ref{Lemma1} to $\gamma$ and $v$: since $W_{i,j}\le \overline W$ for all $i\sim j$, we get
\begin{eqnarray*}
\E^\Q\left(e^{s u_y}\right)
\le
e^{-\gamma s+  2 \gamma^2 q^2 (\overline W+1)\eee(v,v)}
=
e^{-R(0,y)\tilde \gamma s-2\tilde\gamma^2 R(0,y)^2q^2(\overline W+1)\eee(v,v)}
=
e^{-R(0,y)(\tilde \gamma s-2\tilde\gamma^2 q^2(\overline W+1))},
\end{eqnarray*}
since $\eee(v,v)=1/R(0,y)$.
The infimum on $\tilde \gamma$ of the right-hand side is obtained for 
$$\tilde \gamma={s\over 4q^2(\overline W+1)}\le {1\over 2 q^2}.$$ 
Choosing $\tilde \gamma$ as above, it satisfies the condition \eqref{condition}, so that we get 
\begin{eqnarray}\label{explicite}
\E^\Q\left(e^{s u_y}\right)\le e^{-R(0,y){s^2\over 8 q^2(\overline W+1)}}\le e^{-{c_0 s^2\over 8 q^2(\overline W+1)}\ln\vert y\vert}.
\end{eqnarray}
Taking $\eta(s,\overline W):={c_0 s^2\over 8 q^2(\overline W+1)}$ concludes the proof of the lemma.
\begin{rmk}
Note that when $\overline W\to 0$, we cannot get an arbitrary large exponent $\eta(s,\overline W)$. This is rather surprising since, by a different argument, at small $\overline W$ it is known that the field is exponentially localized (see \cite{DS10}). The same phenomenon appears in the proof of Merkl and Rolles of the polynomial localisation of the mixing field of the ERRW (see \cite{merkl2009recurrence}), where a Mermin-Wagner argument is also used.  This is what prevented them to prove recurrence of the 2D-ERRW at strong disorder. Indeed, without extra considerations, one needs an exponent $\eta$ at least larger than 1 to get recurrence.
\end{rmk}

\hfill\break
\noindent
{\it 
{\bf Note}
Gady Kozma and Ron Peled also have a proof of a similar result, see forthcoming \cite{Kozma-Peled-19}. From recent discussion with them, we concluded that our two approaches are rather different. We thank them for communicating an early version of their manuscript. 
\hfill\break
{\bf Acknowledgement} We are grateful to Tyler Helmuth for pointing ref~\cite{McBryan-Spencer77}.}
\bibliography{bibi-Polynomial}

\begin{thebibliography}{10}

\bibitem{angel2014localization}
Omer Angel, Nicholas Crawford, and Gady Kozma.
\newblock {Localization for linearly edge reinforced random walks}.
\newblock {\em Duke Mathematical Journal}, 163(5):889--921, 2014.

\bibitem{basdevant2012continuous}
Anne-Laure Basdevant and Arvind Singh.
\newblock {Continuous-time vertex reinforced jump processes on Galton--Watson
  trees}.
\newblock {\em The Annals of Applied Probability}, 22(4):1728--1743, 2012.

\bibitem{BHS18}
Roland {Bauerschmidt}, Tyler {Helmuth}, and Andrew {Swan}.
\newblock {Dynkin isomorphism and Mermin--Wagner theorems for hyperbolic sigma
  models and recurrence of the two-dimensional vertex-reinforced jump process}.
\newblock {\em arXiv e-prints}, page arXiv:1802.02077, Feb 2018.

\bibitem{collevecchio2009limit}
Andrea Collevecchio.
\newblock {Limit theorems for vertex-reinforced jump processes on regular
  trees}.
\newblock {\em Electron. J. Probab}, 14(66):1936--1962, 2009.

\bibitem{coppersmith1987random}
Don Coppersmith and Persi Diaconis.
\newblock {Random walk with reinforcement}.
\newblock {\em Unpublished manuscript}, pages 187--220, 1987.

\bibitem{davis2004vertex}
Burgess Davis and Stanislav Volkov.
\newblock {Vertex-reinforced jump processes on trees and finite graphs}.
\newblock {\em Probability theory and related fields}, 128(1):42--62, 2004.

\bibitem{diaconis1980finetti}
Persi Diaconis and David Freedman.
\newblock {de Finetti's theorem for Markov chains}.
\newblock {\em The Annals of Probability}, pages 115--130, 1980.

\bibitem{DST14}
Margherita Disertori, Christophe Sabot, and Pierre Tarres.
\newblock {Transience of edge-reinforced random walk}.
\newblock {\em Communications in Mathematical Physics}, 339(1):121--148, 2015.

\bibitem{DS10}
Margherita Disertori and Tom Spencer.
\newblock {Anderson localization for a supersymmetric sigma model}.
\newblock {\em Communications in Mathematical Physics}, 300(3):659--671, 2010.

\bibitem{DSZ06}
Margherita Disertori, Tom Spencer, and Martin~R Zirnbauer.
\newblock {Quasi-diffusion in a 3D supersymmetric hyperbolic sigma model}.
\newblock {\em Communications in Mathematical Physics}, 300(2):435--486, 2010.

\bibitem{Kozma-Peled-19}
Gady Kozma and Ron Peled.
\newblock Power-law decay of weights and recurrence of the two-dimensional
  vrjp.
\newblock {\em private communication}, 2019.

\bibitem{lyons-peres}
Russell Lyons and Yuval Peres.
\newblock {\em Probability on trees and networks}, volume~42 of {\em Cambridge
  Series in Statistical and Probabilistic Mathematics}.
\newblock Cambridge University Press, New York, 2016.

\bibitem{McBryan-Spencer77}
Oliver~A. McBryan and Thomas Spencer.
\newblock On the decay of correlations in {${\rm SO}(n)$}-symmetric
  ferromagnets.
\newblock {\em Comm. Math. Phys.}, 53(3):299--302, 1977.

\bibitem{merkl2009recurrence}
Franz Merkl and Silke~WW Rolles.
\newblock {Recurrence of edge-reinforced random walk on a two-dimensional
  graph}.
\newblock {\em The Annals of Probability}, pages 1679--1714, 2009.

\bibitem{ST15}
Christophe Sabot and Pierre Tarr\`es.
\newblock {{Edge-reinforced random walk, vertex-reinforced jump process and the
  supersymmetric hyperbolic sigma model}}.
\newblock {\em J. Eur. Math. Soc.}, 17(9):2353--2378, 2015.

\bibitem{STZ15}
Christophe Sabot, Pierre Tarr\`es, and Xiaolin Zeng.
\newblock The vertex reinforced jump process and a random {S}chr\"{o}dinger
  operator on finite graphs.
\newblock {\em Ann. Probab.}, 45(6A):3967--3986, 2017.

\bibitem{SZ15}
Christophe Sabot and Xiaolin Zeng.
\newblock A random {S}chr\"{o}dinger operator associated with the vertex
  reinforced jump process on infinite graphs.
\newblock {\em J. Amer. Math. Soc.}, 32(2):311--349, 2019.

\bibitem{Zirnbauer91}
Martin~R Zirnbauer.
\newblock {Fourier analysis on a hyperbolic supermanifold with constant
  curvature}.
\newblock {\em Communications in mathematical physics}, 141(3):503--522, 1991.

\end{thebibliography}
\bibliographystyle{plain}

\end{document}